\documentclass{amsart}
\usepackage{latexsym, amssymb, amsmath}

\newtheorem{thm}{Theorem}[section]

\newtheorem{prop}[thm]{Proposition}
\theoremstyle{remark}
\newtheorem{rmk}[thm]{Remark}
\theoremstyle{definition}

\title{$3$-manifolds with positive flat conformal structure} 

\author{Reiko Aiyama}
\email{aiyama@math.tsukuba.ac.jp}
\address{Department of Mathematics, University of Tsukuba, 
Tsukuba 305-8571, Japan}

\author{Kazuo Akutagawa${}^*$} 
\email{akutagawa@math.is.tohoku.ac.jp} 
\address{Division of Mathematics, GSIS, 
Tohoku University, Sendai 980-8579, Japan} 

\thanks{${}^*$\ 
supported in part by the Grants-in-Aid for Scientific Research (C), 
Japan Society for the Promotion of Science, No.~21540097.} 

\date{March, 2011.}
\begin{document} 
\maketitle
\markboth{$3$-manifolds with positive flat conformal structure}
{Reiko Aiyama and Kazuo Akutagawa}

\begin{abstract}

In this paper, 
we consider a closed $3$-manifold $M$ with flat conformal structure $C$. 
We will prove that, if the Yamabe constant of $(M, C)$ is positive, 
then $(M, C)$ is Kleinian. 
\end{abstract}
\maketitle
  
\section{Introduction and Main Theorem} 
In 1988, Schoen and Yau~\cite{SY} gave a final resolution for the {\it Yamabe Problem} 
(cf.~\cite{Au-Book, LP, Sc-2}). 
In \cite[Proposition~3.3]{SY}, they also proved that 
{\it any closed $n$-\it manifold with flat conformal structure of positive Yamabe constant is Kleinian, 
provided that $n \geq 4$}. 
Moreover, under the assumption that an extended Positive Mass Theorem holds 
(but a proof has not yet appeared), 
they showed that the above assertion still holds even when $n = 3$ 
(see \cite[Proposition~4.4$'$]{SY} and the paragraph just before it). 
On the other hand, there are enormous examples of closed $3$-manifolds with flat conformal structures 
which are not Kleinian (see \cite[Remark~7.4]{GLT}). 

The purpose of this brief note is to prove the above assertion for the remaining case $n = 3$.  

\begin{thm}\label{Main} 
Let $M$ be a closed $3$-manifold with flat conformal structure $C$. 
If its Yamabe constant is positive, 
then $(M, C)$ is Kleinian. 
\end{thm} 

This assertion can be obtained by an argument in the proof of \cite[the second assertion of Theorem~1.4]{Ak}, 
which is a combination of a result \cite[Proposition~4.2]{SY}, 
a positive mass theorem \cite[the first assertion of Theorem~1.4]{Ak} 
(different form the one Schoen and Yau mentioned in \cite{SY}) and 
a classification of $3$-manifolds with positive scalar curvature \cite{GL, Iz-1, Iz-2}. 
Here, we will explicitly give a proof of it. 

The remaining sections are organized as follows. 
Section~2 contains some necessary definitions and preliminary geometric results. 
Section~3 is devoted to the proof of Theorem~\ref{Main}. 

\noindent
{\bf Acknowledgements.} 
The second author would like to thank Hiroyasu Izeki and Gilles Carron 
for helpful discussions and for useful comments respectively.

\section{Preliminaries} 
Let $M$ be a closed $3$-manifold, that is, a compact $3$-manifold without boundary. 
To simplify the presentation and the argument, we always assume that dim~$M = 3$ throughout this paper. 
For each conformal class $C$ on $M$, the {\it Yamabe constant} $Y(M, C)$ of $(M, C)$ 
is defined by 
$$ 
Y(M, C) := \inf_{g \in C}E(g),\qquad E(g) := \frac{\int_M R_g d\mu_g}{{\rm Vol}_g(M)^{1/3}}, 
$$
where $R_g, \mu_g$ and ${\rm Vol}_g(M)$ denote respectively the scalar curvature, the volume element of $g$ 
and the volume of $(M, g)$. It is a finite-valued conformal invariant of $C$. 
The Yamabe constant $Y(M, C)$ is positive if and only if 
there exists a positive scalar curvature metric $g \in C$ (cf.~\cite{Au-Book}). 
A remarkable theorem \cite{Yamabe, Tr, Au, Sc-1, SY} 
of Yamabe, Trudinger, Aubin and Schoen asserts that each conformal class $C$ contains 
a minimizer $\check{g}$ of $E|_C$, called a {\it Yamabe metric} 
(or a {\it solution of the Yamabe Problem}), 
which is of constant scalar curvature 
$$ 
R_{\check{g}} = Y(M, C)\cdot {\rm Vol}_{\check{g}}(M)^{-2/3}. 
$$ 

Let $M_{\infty}$ is an infinite covering of $M$. 
We shall call that the fundamental group $\pi_1(M)$ of $M$ has 
a {\it descending chain of finite index subgroups tending to} $\pi_1(M_{\infty})$ 
if it satisfies the following: 
There exists a family of subgroups $\{\Gamma_i\}_{i \geq 1}$ of $\pi_1(M)$ such that 
\begin{enumerate}  
\item[(i)] each $\Gamma_i$ is finite index in $\pi_1(M)$ with $\Gamma_i \supset \pi_1(M_{\infty})$,  
\item[(ii)] $\pi_1(M) = \Gamma_1 \supsetneq \Gamma_2 \supsetneq \cdots\ \supsetneq \Gamma_i \supsetneq \Gamma_{i + 1} \supsetneq \cdots\ ,$ 
\item[(iii)] $\bigcap_{i=1}^{\infty}\Gamma_i = \pi_1(M_{\infty}).$ 
\end{enumerate} 
Assume that $Y(M, C) > 0$. Take a positive scalar curvature metric $g \in C$ and any point $p \in M$. 
Then, there exists the {\it normalized Green's function} $G_p$ for $L_g$ with a pole at $p$, that is, 
$$ 
L_g G_p = c_0\cdot \delta_p\quad {\rm on}\ \ M\quad {\rm on}\quad \lim_{q \to p}{\rm dist}(q, p) G_p(q) = 1. 
$$ 
Here, $L_g := -8 \Delta_g + R_g, c_0 > 0$ and $\delta_p$ stand respectively for the {\it conformal Laplacian}, 
a specific universal positive constant and the Dirac $\delta$-function at $p$. 
Assume also that the covering $P_{\infty} : M_{\infty} \rightarrow M$ is normal. 
Let $g_{\infty}$ denote the lift of $g$ to $M_{\infty}$, 
and $p_{\infty}$ a point in $M_{\infty}$ with $P_{\infty}(p_{\infty}) = p$. 
Then, there exists uniquely also a {\it normalized minimal positive Green's function} $G_{\infty}$ on $M_{\infty}$ 
for $L_{g_{\infty}} := -8 \Delta_{g_{\infty}} + R_{g_{\infty}}$ with pole at $p_{\infty}$  (cf.~\cite{SY}), 
which satisfies the following 
$$ 
(P_{\infty})^{\ast} G_p = \sum_{\gamma \in \mathcal{G}}G_{\infty}\circ \gamma\quad {\rm on}\ \ M_{\infty}. 
$$ 
Here, $\mathcal{G}$ stands for the group of deck transformations 
for the normal covering $M_{\infty} \rightarrow M$. 
Set 
$$ 
g_{\infty, AF} := G_{\infty}^4\cdot g_{\infty}\quad {\rm on}\ \ M_{\infty}^{\ast} := M_{\infty} - \{p_{\infty}\}. 
$$ 
Then, $g_{\infty, AF}$ defines a scalar-flat, asymptotically flat metric on $M_{\infty}^{\ast}$ (cf.~\cite{LP}). 
Note that this asymptotically flat $3$-manifold $(M_{\infty}^{\ast}, g_{\infty, AF})$ has 
{\it infinitely many} singularities created by the ends of $M_{\infty}^{\ast}$. 
However, the mass $\frak{m}_{ADM}(g_{\infty, AF})$ of $(M_{\infty}^{\ast}, g_{\infty, AF})$ 
can be defined in the usual way (cf.~\cite{B}). 
Note also that the positive mass theorem for asymptotically flat $3$-manifolds with singularities 
does not always hold (see \cite[Remark~1.5-(2)]{Ak} for instance). 

With these understanding, the following positive mass theorem holds 
as a special case of \cite[the first assertion of Theorem~1.4]{Ak}: 

\begin{prop}\label{Prop-1} 
Let $(M, C)$ be a closed $3$-manifold with $Y(M, C) > 0$. 
Let $(M_{\infty}, g_{\infty})$ be a normal infinite Riemannian covering of $(M, g)$ 
such that $\pi_1(M)$ has a descending chain of finite index subgroups tending to $\pi_1(M_{\infty})$, 
where $g \in C$ is a positive scalar curvature metric and $g_{\infty}$ is its lift to $M_{\infty}$. 
For any point $p_{\infty} \in M_{\infty}$, 
let $G_{\infty}$ denote the normalized minimal positive Green's function on $M_{\infty}^{\ast}$ with pole at $p_{\infty}$. 
Then, the asymptotically flat $3$-manifold $(M_{\infty}^{\ast}, g_{\infty, AF})$ has nonnegative mass 
$$ 
\frak{m}_{ADM}(g_{\infty, AF}) \geq 0. 
$$ 
\end{prop} 

\begin{rmk}\label{spin} 
Assume that $M = \#\ell (S^1 \times S^2)$ for $\ell \geq 2$ 
and $M_{\infty}$ is its universal covering. 
Note that $M_{\infty}$ is spin. 
For each small $\sigma > 0$, 
consider the complete metric $g_{\sigma, AF} := (G_{\infty} + \sigma)^4\cdot g_{\infty}$ 
with $R_{g_{\sigma, AF}} \geq 0$ on $M_{\infty}^*$ 
(cf.~\cite[Proposition~4.4$'$]{SY}).  
Then, only one end of $(M_{\infty}^*, g_{\sigma, AF})$ is asymptotically flat 
and the other infinitely many ends are merely complete. 
For the authors, it is not clear whether Witten's approach \cite{Witten} (cf.~\cite{PT}) 
to Positive Mass Theorem is still valid for $(M_{\infty}^*, g_{\sigma, AF})$. 
Hence, we will use here Proposition~\ref{Prop-1} for the proof. 
\end{rmk} 

A conformal $3$-manifold $(M, C)$ is is said to be {\it locally conformally flat} 
if, for any point $p \in M$, there exists a metric $\overline{g} \in C$ 
such that $\overline{g}$ is flat on some neighborhood of $p$. 
A conformal class $C$ on $M$ is called a {\it flat conformal structure} if 
$(M, C)$ is locally conformally flat. 
In \cite{Ku}, Kuiper proved that, for a simply connected locally conformally flat $3$-manifold $(X, C')$, 
there is a conformal immersion into $(S^3, C_0)$ called {\it developing map}, 
which is unique up to composition with a M\"obius transformation of $(S^3, C_0)$. 
Therefore, the universal covering of a locally conformally flat manifold $(M, C)$ admits a developing map. 
Here, $(S^3, C_0)$ denotes the $3$-sphere $S^3$ 
with the conformal class $C_0 := [g_0]$ of the standard metric $g_0$ 
of constant curvature one. 
$(M, C)$ is called {\it Kleinian} if $(M, C)$ is conformal to $\Omega/\Gamma$ 
for some open set $\Omega$ of $S^3$ and some discrete subgroup $\Gamma$ 
of the conformal transformation group ${\rm Conf}(S^3, C_0)$, 
which leaves $\Omega$ invariant and acts freely and properly discontinuously on $\Omega$. 
Note that, if the developing map of the universal covering of a locally conformally flat manifold 
$(M, C)$ is injective, then $(M, C)$ is Kleinian. 

With these understanding, the following criterion also holds 
as a special case of \cite[Proposition~4.2]{SY}: 

\begin{prop}\label{Prop-2} 
Let $(M, C)$ be a closed $3$-manifold with $Y(M, C) > 0$,  
and $(\widetilde{M}, \widetilde{g})$ the universal Riemannian covering of $(M, g)$, 
where $g \in C$ is a positive scalar curvature metric. 
For any point $\widetilde{p} \in \widetilde{M}$, 
let $\widetilde{G}$ denote the normalized minimal positive Green's function on $\widetilde{M}$ 
for $L_{\widetilde{g}}$ with pole at $\widetilde{p}$, 
and $(\widetilde{M} - \{\widetilde{p}\}, \widetilde{g}_{AF} = \widetilde{G}^4\cdot \widetilde{g})$ 
the asymptotically flat $3$-manifold as above. 
If the mass $\frak{m}_{ADM}(\widetilde{g}_{AF})$ is nonnegative, 
then the developing map of $(\widetilde{M}, [\widetilde{g}])$ is injective. 
In particular, $(M, C)$ is Kleinian. 
\end{prop} 

\begin{rmk} 
We remark that the mass $\frak{m}_{ADM}(\widetilde{g}_{AF})$ is equal to 
the ADM energy $E$ of $(\widetilde{M} - \{\widetilde{p}\}, \widetilde{g}_{AF})$ 
appeared in \cite[page~64]{SY} up to a positive constant. 
\end{rmk}

\section{Proof of Main Theorem} 
\begin{proof}[Proof of Theorem~\ref{Main}] 
Consider the universal covering $\widetilde{M}$ of $M$ 
and denote the lift of the flat conformal structure $C$ by $\widetilde{C}$. 
If $|\pi_1(M)| < \infty$, 
then $(\widetilde{M}, \widetilde{C})$ is conformal to $(S^3, C_0)$ 
by Kuiper's Theorem \cite{Ku}. 
Hence, $(M, C)$ is Kleinian. 
From now on, we assume that $|\pi_1(M)| = \infty$, that is, 
the degree of the covering map $P : \widetilde{M} \rightarrow M$ is infinite. 

Take a unit-volume Yamabe metric $g \in C$, 
and consider its lift $\widetilde{g} \in \widetilde{C}$ to $\widetilde{M}$. 
Note that $R_{\widetilde{g}} = R_g = Y(M, C) > 0$. 
Take any base points $p \in M, \widetilde{p} \in \widetilde{M}$ satisfying 
$P(\widetilde{p}) = p$, and fix them. 
Then, let $\widetilde{G}$ denote the normalized minimal positive Green function on $\widetilde{M}$ 
for $L_{\widetilde{g}}$ with pole at $\widetilde{p}$, 
and the mass $\frak{m}_{\rm ADM}(\widetilde{g}_{AF})$ of the asymptotically flat $3$-manifold 
$(\widetilde{M} - \{\widetilde{p}\}, \widetilde{g}_{AF} := \widetilde{G}^4\cdot \widetilde{g})$. 

Suppose that 
$$ 
\frak{m}_{\rm ADM}(\widetilde{g}_{AF}) \geq 0.  
$$ 
Recall that we can choose the base point $\widetilde{p} \in \widetilde{M}$ arbitrarily. 
It then follows from Proposition~\ref{Prop-2} that 
the developing map of $(\widetilde{M}, \widetilde{C})$ is injective, 
and hence $(M, C)$ is Kleinian.  
In this case, especially $\frak{m}_{\rm ADM}(\widetilde{g}_{AF}) = 0$. 
Therefore, it is enough to show $\frak{m}_{\rm ADM}(\widetilde{g}_{AF}) \geq 0$. 

By combining \cite[Theorem~8.1]{GL} (cf.~\cite{Hem}) with $Y(M, C) > 0$, 
(replacing $M$ by its orientable double covering if necessary) 
$M$ can be decomposed uniquely into {\it prime} closed $3$-manifolds 
$$ 
M = N_1 \# \cdots \# N_{\ell_1} \# \ell_2(S^1 \times S^2), 
$$ 
where $\pi_1(N_i)$ is finite for $i = 1, \cdots, \ell_1$ 
and $\ell_1, \ell_2$ are nonnegative integers. 
By applying the $C$-prime decomposition theorem 
for closed $3$-manifolds with flat conformal structures \cite{Iz-1, Iz-2} to $(M, C)$, 
there exists a flat conformal structure $C_i$ on each $N_i~(i = 1, \cdots, \ell_1)$. 
Then, Kuiper's Theorem~\cite{Ku} again implies that 
each $(N_i, C_i)$ is a non-trivial quotient of $(S^3, C_0)$. 
After taking an appropriate finite covering $M'$ of $M$, 
we have 
$$ 
M' = \# \ell(S^1 \times S^2)\quad \textrm{for some}\ \ \ell \geq 1. 
$$ 
Recall that $\widetilde{M}$ is the infinite universal covering of $M$. 
Then, there exists (uniquely) an infinite universal covering $\widetilde{M} \rightarrow M'$. 
Moreover, since $\pi_1(M')$ is a finitely generated free group, 
it has a descending chain of finite index subgroups 
tending to $\pi_1(\widetilde{M}) = \{e\}$. 
Let $g'$ be the lifting of $g$ to $M'$. 
Applying Proposition~\ref{Prop-1} to 
the normal infinite Riemannian covering $(\widetilde{M}, \widetilde{g}) \rightarrow (M', g')$, 
we have that 
$$ 
\frak{m}_{\rm ADM}(\widetilde{g}_{AF}) \geq 0. 
$$ 
This completes the proof of Theorem~\ref{Main}. 
\end{proof}  

\begin{rmk} 
Even if we replace the positivity $Y(M, C) > 0$ in Theorem~\ref{Main} 
by the nonnegativity $Y(M, C) \geq 0$, 
it seems that the same conclusion still holds. 
More precisely, we propose the following (cf.~\cite{Bou, Ko}).

{\bf Conjecture.}\ \ 
{\it Let $M$ be a closed $3$-manifold with flat conformal structure $C$. 
If its Yamabe constant is zero, 
then either of the following $(1)$ or $(2)$ holds$:$\\ 
$(1)$\ \ There exists a flat metric $\overline{g} \in C$. \\ 
$(2)$\ \ There exists a smooth family $\{g_t\}_{0 \leq t \leq 1}$ of locally conformally flat metrics on $M$ 
such that $g_0 \in C$ and $Y(M, [g_1]) > 0$.  
} 

In the case $(1)$, the universal covering $(\widetilde{M}, \widetilde{C})$ of $(M, C)$ 
is conformal to $(S^3 - \{p_N\}, C_0)$ where $p_N := (1, 0, 0, 0) \in S^3$, 
and hence $(M, C)$ is Kleinian. 
In the case $(2)$, Theorem~\ref{Main} implies that $(M, [g_1])$ is Kleinian. 
The argument in Proof of Theorem~\ref{Main} also implies that there exists a torsion free subgroup $\Gamma$ of finite index in $\pi_1(M)$ 
such that $\Gamma$ is either a trivial group or a non-trivial finitely generated free group. 
Then, the {\it virtual cohomological dimension} ${\rm vcd}~\pi_1(M)$ of $\pi_1(M)$ is either $0$ or $1$ (see \cite{Br-Book}). 
Therefore, $(M, [g_1])$ is a closed Kleinian $3$-manifold with ${\rm vcd}~\pi_1(M) < 3$. 
The quasiconformal stability of Kleinian groups \cite[Theorem~2]{Iz-3} implies that 
any flat conformal structure on $M$ 
which is a smooth deformation of $[g_1]$ is also Kleinian, particularly $C$ is too.  
\end{rmk}

\bibliographystyle{amsbook}

\vspace{10mm} 

\end{document}